\documentclass{amsart}

\usepackage{amsmath}
\usepackage{amsthm}
\usepackage{amssymb}
\usepackage{amsbsy}
\usepackage{amsfonts}
\usepackage{amstext}
\usepackage{amscd}
\usepackage{graphicx}

\numberwithin{equation}{section}
\theoremstyle{plain}
\newtheorem{thm}{Theorem}[section]
\newtheorem{prop}[thm]{Proposition}
\newtheorem{cor}[thm]{Corollary}

\theoremstyle{definition}
\newtheorem{exa}[thm]{Example}

\newtheorem{rem}[thm]{Remark}
\newtheorem{defi}[thm]{Definition}

\newcommand{\real}{\mathbb{R}}

\newcommand{\E}{\mathbb{E}}
\newcommand{\R}{\mathbb{R}}
\newcommand{\Pb}{\mathbb{P}}

\begin{document}
\title[Distance in 4th moment theorem]{Convergence of the Fourth Moment and Infinite Divisibility: Quantitative estimates. }

\author{Octavio Arizmendi}
\author{Arturo Jaramillo}
\address{Research Center for Mathematics, CIMAT, Aparatado Postal 402, Guanajuato, GTO, 36240, Mexico}

\date{\today}
\maketitle

\begin{abstract}
We give an estimate for the Kolmogorov distance between an  infinitely divisible distribution (with mean zero and variance one) and the standard Gaussian distribution in terms of the difference between the fourth moment and 3. In a similar fashion we give an estimate for the Kolmogorov distance between a freely infinitely divisible distribution and the Semicircle distribution in terms of the difference between the fourth moment and 2.  
\end{abstract}

\section{Introduction and Statement of Results}

In recent years, many interest has been put in the characterization
of those chaotic sequences $\{X_n;n\geq1\}$ verifying a Central Limit Theorem (CLT), that is,
such that $X_n$ converges in distribution to $\mathcal{N}(0,1)$ (as $n\rightarrow\infty$), where $\mathcal{N}(0,1)$
denotes a centered Gaussian law with unit variance. A solution to this problem was first given by Nualart and Peccati in
the form of the following ``fourth moment theorem".

\begin{thm}[\cite{NuPe}] \label{T1} 
Let $\{X_{n};n\geq1\}$ be a sequence of multiple Wiener-It\^o integrals of the form $X_n=I^W_m(f_n)$, for kernels $f_{n}\in L^{2}(\R_{+}^{m},\text{d}t)$ ($\text{d}t$ denotes the Lebesgue measure), such that $E[X_n^2] \rightarrow 1$.  Then the following are equivalent
\begin {enumerate}
\item $ E[X_n^4] \rightarrow 3$
\item $\mu_{X_n}\rightarrow \mathcal{N}(0,1)$
\end{enumerate}
\end{thm}

Since this seminal  work,  a lot of effort has been devoted in finding distributions  other than in a fixed chaos for which a ``fourth moment theorem" would still hold. See the survey \cite{NoPe3} and the monograph \cite{NoPe4} for details and references.  More recent developments can be found in the webpage maintained by Ivan Nourdin 
$$https://sites.google.com/site/malliavinstein/home$$

As an important example for us, in the free probability setting ,  it was proved by Kemp et al. \cite{KNPS} that the Nualart-Peccati criterion also holds for
the free Brownian motion $\{S_t;t\geq0\}$ and its multiple Wigner integrals $I_m^S(f)$.  

\begin{thm}[\cite{KNPS}] \label{T2}
Let $\{X_n;n\geq1\}$ be a sequence of multiple Wigner integrals of the form $X_n=I_m^S(f_n)$ in a fixed $m$-chaos with $E[X_n^2] \rightarrow 1$ and denote by $\mathcal{S}(0,1)$
 a centered Semicircle law with unit variance.  Then the following are equivalent 
\begin{enumerate}
\item $ E[X_n^4] \rightarrow 2$
 \item $\mu_{X_n}\rightarrow \mathcal{S}(0,1)$
\end{enumerate}
\end{thm}

More recently, in \cite{Ar1} the first author proved analogous results to Theorem \ref{T1}  and  Theorem \ref{T2} in the setting of infinitely divisible laws.  Let $ID(*)$ and $ID(\boxplus)$ denote the classes of probability measures which are infinitely divisible with respect to classical convolution  $*$  and free convolution $\boxplus$, respectively. 

\begin{thm}[\cite{Ar1}]\label{T3} Let $\{\mu_n=\mu_{X_n};n\geq1\}$ be a sequence of probability measures with variance $1$ and mean $0$ such that $\mu_n\in ID(*)$. If $ E[X_n^4] \rightarrow 3$
then $\mu_{X_n}\rightarrow \mathcal{N}(0,1)$.
\end{thm}

\begin{thm}[\cite{Ar1}]\label{T4} Let $\{\mu_n=\mu_{X_n};n\geq1\}$ be a sequence of probability measures with variance $1$ and mean $0$ such that $\mu_n\in ID(\boxplus)$. If $ E[X_n^4] \rightarrow 2$
then $\mu_{X_n}\rightarrow \mathcal{S}(0,1)$. 
\end{thm}

In this note we give quantitative versions of Theorems \ref{T3} and \ref{T4}. That is, we give precise estimates for the Kolmogorov distance between an infinitely divisible measure  $\mu$ and $ \mathcal{N}(0,1)$ ( resp. $ \mathcal{S}(0,1)$)  in terms of the fourth moment.

\begin{thm}\label{T5} Let $\mu\in ID(*)$ be a probability measures with variance $1$ and mean $0$.  Then 
$$d_{kol}(N(0,1), \mu)\leq C \sqrt{m_4-3},$$
where $m_4$ denotes the fourth moment of $\mu$ and C is a universal constant.
\end{thm}

\begin{thm}\label{T6} Let $\mu\in ID(\boxplus)$ be a probability measures with variance $1$ and mean $0$.  Then 
$$d_{kol}(S(0,1), \mu)\leq K \sqrt{m_4-2},$$
where $m_4$ denotes the fourth moment of $\mu$ and K is a universal constant.
\end{thm}

The proof of Theorems \ref{T5} and \ref{T6} relies on the Berry-Esseen Theorem and its free version (See Section 2.3). Furthermore, we can prove slightly stronger versions of Theorems \ref{T3} and \ref{T4}, changing infinite divisibility by just $n$-divisibility.

\begin{thm}\label{T7} Let $\{\mu_n=\mu_{X_n};n\geq1\}$ be a sequence of probability measures with variance $1$ and mean $0$ such that $\mu_n$ is $n$-divisible with respect to classical convolution. If $ E[X_n^4] \rightarrow 3$
then $\mu_{X_n}\rightarrow \mathcal{N}(0,1)$.
\end{thm}

\begin{thm}\label{T8} Let $\{\mu_n=\mu_{X_n};n\geq1\}$ be a sequence of probability measures with variance $1$ and mean $0$ such that $\mu_n$ is $n$-divisible with respect to free convolution. If $ E[X_n^4] \rightarrow 2$
then $\mu_{X_n}\rightarrow \mathcal{S}(0,1)$.
\end{thm}

Finally, we want to point out that explicit bounds (for the total variation)  in the Gaussian approximations of random variables in a fixed Wiener chaos were given by Nourdin and Peccati in \cite{NoPe1,NoPe2}.

The paper is organized as follows. In Section 2 we give some preliminaries. In Section 3 we prove Theorems \ref{T5}-\ref{T8}.
In Section 4, we give few examples of application of the main results of this paper. Finally, we include an appendix where we show inequalities on the fourth moment for $N$-divisible measures which explain somehow the role of the Gaussian and Semicircle distribution as extremal points among the class of infinitely divisible measures.

\section{Preliminaries}

In this section we give some basic preliminaries on cumulants and free cumulants.  The reader familiar with these objects may skip this parts. We also present the Berry-Esseen theorem and its free version.
\subsection{Cumulants}

Let $\mathcal{M}$ denote the set of Borel probability measures on $\real$. We say that a measure $\mu $ has \textit{all moments} if $m_{k}(\mu )=\int_{\mathbb{R}}t^{k}\mu (\mathrm{d}t)<\infty ,$ for each
integer $k\geq 1$. 

Recall that the classical convolution of two probability measures $\mu_1,\mu_2$ on $\mathbb{R}$ is defined as the probability measure $\mu_1*\mu_2$ on $\mathbb{R}$ such that
$\mathcal{C}_{\mu_1*\mu_2}(t)=\mathcal{C}_{\mu_1}(t)+\mathcal{C}_{\mu_2}(t),t \in \real,$
where $\mathcal{C}_{\mu}(t) = \log \hat{\mu}(t),$ with $\hat{\mu}(t)$ the characteristic function of $\mu$.

Let $\mu \in \mathcal{M}$ be a probability measure with all its moments.  The coefficients $c_n = c_n (\mu)$ in the series expansion $$\mathcal{C}_{\mu}(t) =\sum^\infty_{n=1}\frac{c_n}{n!} t^n$$
are called (classical) cumulants or semi-invariants.

Since the cumulant transform $\mathcal{C}_{\mu }$ linearizes classical convolution, then the cumulants are also additive with respect to the convolution $\mu _{1}* \mu _{2}$ 
\begin{equation}
c_{n}(\mu _{1}* \mu _{2})=c_{n}(\mu _{1})+c_{n} (\mu_{2})
\end{equation}
and 
\begin{equation}\label{aux2}
c_{n}(\mu ^{*t})=tc_{n}(\mu ). 
\end{equation}

The relation between the cumulants and the moments is given in terms of the set  $P(n)$ of partitions of $\left\{ 1,\dots,n\right\}$, be the so-called moment-cumulant formula, 
\begin{equation}\label{moment-cumulant formula}
m_{n}(\mu )=\sum\limits_{\mathbf{\pi }\in P(n)}c_{\pi }(\mu ), 
\end{equation}
where $\pi \rightarrow c_{\pi }$ is the multiplicative extension of the cumulants to
partitions, that is
\begin{equation*}
c_{\pi }:=c_{| V_{1}| }\cdots c_{| V_{r}| }\qquad
\text{for}\qquad \pi =\{V_{1},...,V_{r}\}\in NC(n) \text{.}
\end{equation*}

The first moments are written in terms of cumulants as follows:
\begin{eqnarray}
m_1&=&c_1\nonumber\\
m_2&=&c_2+c_1^2\label{cumulants}\\
m_3&=&c_3+3c_2c_1+c_1^3\nonumber\\
m_4&=&c_4+4c_3c_1+3c_2^2+6c_{2}c_{1}^2+c_{1}^{4}.\nonumber 
\end{eqnarray}

Let $X$ be a random variable with distribution $\mu$.  We say that $X$ has absolute moments up to  order $n$ if $m_n(|X|)=E(|X|^n)=\int |x|^n \mu(dx) <\infty$. For a random variable $X$ with absolute moments up to order $n$ we can still define the cumulant of order $n$ by the
moment-cumulant formula 
\begin{equation}
m_n(X)=E(X^n)=\sum_{\pi\in P(n)}c_{\pi}(X).
\end{equation}

It will be important for us that cumulants are linear with respect to addition of random variables. That is, for independent random variables $X,Y$,  and $\lambda\in\mathbb{C}$,

\begin{equation}\label{aux1}c_n(X+\lambda Y)=c_n(X)+\lambda c_n(Y).
\end{equation}

\subsection{Non-Commutative Probability Spaces}

A $C^*$\textit{-probability space} is a pair $(\mathcal{A},\tau)$, where $\mathcal{A}$ is a unital $C^*$-algebra and $\tau:\mathcal{A}\to\mathbb{C}$ is a positive unital linear functional. The elements of $\mathcal{A}$ are called (non-commutative) random variables. An element $a\in\mathcal{A}$ such that $a=a^*$ is called self-adjoint.

The functional $\tau$ should be understood as the expectation in classical probability. For $a_1,\dots,a_k\in \mathcal{A}$, we will refer to the values of $\tau(a_{i_1}\cdots a_{i_n})$, $1\leq i_1,...,i_{n}\leq k$, $n\geq1$, as the \textit{mixed moments} of $a_1,\dots,a_k$.

For any self-adjoint element $a\in\mathcal{A}$ there exists a unique probability measure with compact support $\mu_a$ (its distribution) with the same moments as $a$, that is, $$\int_{\mathbb{R}}x^{k}\mu_a (dx)=\tau (a^{k}), \quad \forall k\in \mathbb{N}.$$

Even if we know the individual distribution of two self-adjoint elements $a,b\in \mathcal{A}$, their joint distribution (mixed moments) can be quite arbitrary, unless some notion of independence is assumed to hold between $a$ and $b$.  Here, we will work with free independence.

\begin{defi} Let $(A_n)_{n\geq 1}$ be a sequence of subalgebras of $\mathcal{A}$ and, for $a\in \mathcal{A}$, denote by $\bar{a}:=a-\tau(a)$. 
We say that $(A_n)_{n\geq 1}$ are \textit{freely independent} or \text{free} if
\begin{equation}
\tau(\bar{a}_1\bar{a}_2 \cdots \bar{a}_k)=0,
\end{equation}
whenever $k\geq 1$, $a_1,\dots a_k\in \mathcal{A}$, are such that  $a_i\in A_{j(i)}$, $1\leq i\leq k$ and $j(i)\neq j(i+1)$.
\end{defi}

\subsubsection{Free convolution}
Free convolution was defined in \cite{Voi85} for probability measures with compact support and later extended in \cite{Maa} for the case of finite variance, and in \cite{BeVo} for the general unbounded case.

The upper half-plane and the lower half-plane are respectively denoted as $\mathbb{C}^+$ and $\mathbb{C}^-$. 
Let $G_\mu(z) = \int_{\real}\frac{\mu(dx)}{z-x}$ $(z \in \mathbb{C}^+)$ be the Cauchy transform of $\mu \in \mathcal{M}$ and 
$F_\mu(z)$ its reciprocal $\frac{1}{G_\mu(z)}$. 

It was proved in Bercovici and Voiculescu \cite{BeVo} that there are positive numbers $\eta$
and $M$ such that $F_\mu$ has a right inverse $F_\mu^{-1}$ defined on the region
$\Gamma_{\eta,M}:= \{z\in\mathbb{C}^+; |Re(z)| < \eta Im(z),|z|>M\}.$

The Voiculescu transform of $\mu$ is defined by $\phi _{\mu }\left( z\right) =F_{\mu }^{-1}(z)-z,$
on any region of the form $\Gamma_{\eta,M}$ where $F^{-1}_{\mu}$ is defined; see \cite{BeVo}. The free cumulant
transform or $R$-transform is a variant of $\phi_\mu$  defined as
$\mathcal{C}_{\mu}^{\boxplus}(z)=R_{\mu }\left( z\right) =z\phi _{\mu }(\frac{1}{z})$
for $z$ in a domain $D_\mu\subset\mathbb{C}^-$ such that $1/z\in\Gamma_{\eta,M}$ where $F_{\mu}^{-1}$ is defined.

The \emph{free additive convolution} of two probability measures $\mu_1,\mu_2$ on $\mathbb{R}$ is the
probability measure $\mu_1\boxplus\mu_2$ on $\mathbb{R}$ such that 
$$\phi_{\mu_1\boxplus\mu_2}(z) = \phi_{\mu_1}(z) + \phi_{\mu_2}(z), \quad \text{for } z\in\Gamma_{\eta_1,M_1}\cap \Gamma_{\eta_2,M_2}$$
or,
equivalently, 
$$\mathcal{C}_{\mu_1\boxplus\mu_2}^{\boxplus}(z) = \mathcal{C}_{\mu_1}^{\boxplus}(z) + \mathcal{C}_{\mu_2}^{\boxplus}(z), \quad \text{for } z\in D_{\mu_1}\cap D_{\mu_1}.$$

Free additive convolution corresponds to the sum of free random variables: $\mu_a\boxplus\mu_b=\mu_{a+b}$, for $a$ and $b$ free random variables.

\subsection{Free cumulants}

Free cumulants were introduced by Speicher \cite{Sp94} in his combinatorial approach to Free Probability.
Let $\mu \in \mathcal{M}$ be a probability measure with all its moments. The free cumulants are the coefficients $k_{n}=k_{n}(\mu )$ in the series expansion
\begin{equation*}
\mathcal{C}_{\mu }^{\boxplus }\mathcal{(}z)=\sum\nolimits_{n=1}^{\infty }k_{n}(\mu )z^{n}.
\end{equation*}

Since, the cumulant transform $\mathcal{C}_{\mu }^{\boxplus }$ linearizes additive free convolution, then the free cumulants
also additive with respect to the free convolution $\mu _{1}\boxplus \mu _{2}$ 
\begin{equation*}
k_{n}(\mu _{1}\boxplus \mu _{2})=k_{n}\mathcal{(}\mu _{1})+k_{n} \mathcal{(}\mu
_{2})
\end{equation*}
and 
\begin{equation*}
k_{n}(\mu ^{\boxplus t})=tk_{n}(\mu ). 
\end{equation*}

The main object to describe the relation between the free cumulants and the moments is the set of non-crossing partitions of $\left\{ 1,\dots
,n\right\} $, denoted by $NC\left( n\right)$.   We will identify partitions $\pi$ on the set $\left\{ 1,\dots
,n\right\}$ with equivalence relations $\sim_{\pi}$ such that $a\sim_{\pi}b$ iff $s,b\in A$, for $A\in\pi$.  We say that a partition $\pi$ is \textbf{non-crossing } if $a\sim _{\pi }c$ , $b\sim
_{\pi }d\Rightarrow a\sim _{\pi }b\sim _{\pi }c$ $\sim _{\pi }d$, \ for all $
1\leq a<b<c<d\leq n.$ 

The so-called moment-cumulant formula of Speicher \cite{Sp94} gives a relation between  moments and free  cumulants.
\begin{equation}\label{free moment-cumulant formula}
m_{n}(\mu )=\sum\limits_{\mathbf{\pi }\in NC(n)}k_{\pi }(\mu ), 
\end{equation}
where $\pi \rightarrow k_{\pi }$ is the multiplicative extension of the free cumulants to
non-crossing partitions, that is
\begin{equation*}
k_{\pi }:=k_{| V_{1}| }\cdots k_{| V_{r}| }\qquad
\text{for}\qquad \pi =\{V_{1},...,V_{r}\}\in NC(n) \text{.}
\end{equation*}

The first moments are written in terms of cumulants as follows:
\begin{eqnarray}\label{freemomentcumulants}
m_{1} &=&k_{1} \nonumber\\
m_{2} &=&k_{2}+k_{1}^{2}\\
m_{3} &=&k_{3}+3k_{2}k_{1}+k_{1}^{3}\nonumber \\
m_{4} &=&k_{4}+4k_{3}k_{1}+2k_{2}^{2}+6k_{2}k_{1}^{2}+k_{1}^{4}\nonumber \end{eqnarray}

Similarly as for the classical case, free cumulants are linear with respect to addition of free random variables. That is, for free random variables $X,Y$,  and $\lambda\in\mathbb{C}$,

\begin{equation}k_n(X+\lambda Y)=k_n(X)+\lambda k_n(Y).
\end{equation}

\subsection{Classical and free Berry-Esseen Theorem}

 Let $Y,Z$ two random variables with values in $\R$ with distribution functions $F_{Y},F_{Z}$. The Kolmogorov distance between $F_Y:=\Pb[Y\leq x]$ and $F_Z:=\Pb[Z\leq x]$ is defined by
$$d_{Kol}(F_Y,F_Z)=\sup_{x\in \R}\left|F_{Y}(x)-F_{Z}(x)\right|.$$

Denote by $\Phi(x)$ the distribution function of the standard Gaussian random variable. Let $\{X_{n}; n\geq1\}$ be a sequence of independent and indentically distributed random variables with $m_{1}(X_{1})=0,m_{2}(X_{1})=\sigma^2$, and $\E[|X_{1}|^{3}]<\infty$. Define 
$$Y_{n}:=\frac{1}{\sigma\sqrt{n}}\sum_{k=1}^{n}X_{k}.$$
The Berry-Esseen Theorem provides the following bound for the error of the gaussian approximation to the distribution function $F_{n}(x):=\Pb[Y_{n}\leq x]$;
\begin{eqnarray}\label{berry-esseen}
d_{Kol}(\Phi,F_{n}) 
  &\leq& \frac{C}{\sigma^{3}\sqrt{n}}\E[|X_{1}|^{3}],  
\end{eqnarray}
where $C$ is an absolut constant smaller than 0.4748.\\

The analogue of the Berry-Esseen theorem for free random variables was proved by Chistyakov and G\"otze \cite{ChyGo} and is as follows. Let $\{X_{n};n\geq1\}$ be a sequence of \emph{freely} independent and indentically distributed random variables with mean zero and variance 1, and let $Y_{n}:=\frac{1}{\sqrt{n}}\sum_{k=1}^{n}X_{k}$. Denote by $\mathcal{S}$ the distribution function of a standard Semicirle random variable. If the fourth moment of $X_{n}$ exists, then we have the following estimate for the error of the semi-circular approximation to $Y_{n}$;
\begin{eqnarray}\label{berry-esseen-free}
d_{Kol}(\mathcal{S},F_{Y_n}) 
  &\leq& \frac{K}{\sqrt{n}}(|m_{3}| + \sqrt{m_{4}}),  
\end{eqnarray}
where $K > 0$ is an absolut constant, and $m_3,m_4$ denote the third and fourth moment of $X_{1}$.

\section{Main Theorems} 

In this section we prove Theorems \ref{T5} and \ref{T6}. In fact,  we will prove stronger versions which include $N$-divisible measures from where we can deduce also Theorems  \ref{T7} and \ref{T8}.

\subsection{Classical convolution}

\begin{defi} Let  $\mu$ be a probability measure on $\mathbb{R}$.
\begin{enumerate} \item The measure $\mu$ is called $n$-divisible with respect to classical convolution $*$  if it  is the $n$-fold convolution of another probability measure. That is $\mu=\mu_n^{*n}$, for some $\mu_n$.  
\item The measure $\mu$ is called \emph{infinitely divisible} with respect to classical convolution $*$ if it is $n$-divisible for all $n\in\mathbb{N}$.
\end{enumerate}
\end{defi}

During this section, divisibility will be understood in the classical sense, that is, with respect to the classical convolution.

\begin{thm}\label{t:classical_kolmogorov}
Let $\mu$ be a classical $N$-divisible probability measure with mean 0 and variance 1. Assume that $m_{4}(\mu):=\int x^{4}\mu(dx)<\infty$. Then we have the following estimate for the error of the Gaussian approximation to $\mu$
\begin{align}\label{eq:mainc}
d_{Kol}(\mu,\Phi)
  &\leq  C\sqrt{m_{4}(\mu)-3+\frac{3}{N}},
\end{align}
where $\Phi(x):=\int_{-\infty}^{x}\frac{1}{\sqrt{2\pi}}e^{-\frac{s^{2}}{2}}\text{d}s$ and $C$ is the universal constant from the Berry-Esseen Theorem.
\end{thm}

\begin{rem}In Theorem \ref{t:classical_kolmogorov} we implicitly assume that $m_4(\mu)-3+3/N$ is positive when $\mu$ is $N$-divisible. In Section \ref{kurt} we will show that this is the case and furthermore characterize when the mininum $m_4-3$ is achieved. 

\end{rem}

If we apply the previous theorem to a sequence of $N_n$-divisible measures, we deduce the following corollary which implies Theorem \ref{T7} when $N_n=n$.

\begin{cor}\label{maincorclass}
Let $\{\mu_{n};n\geq1\}$ be a sequence of $N_{n}$-divisible probability measures with variance 1 and mean 0. If $N_{n}\rightarrow\infty$ and $\int s^{4}\mu_{n}(ds)\rightarrow 3$ as $n\rightarrow\infty$, then $\mu_{n}$ converges in distribution to a standard gaussian measure. 

\end{cor}

Moreover, letting the limit as $N\to\infty$ in Theorem \ref{t:classical_kolmogorov} we get Theorem \ref{T5}.

\begin{cor}

Let $\mu$ be infinitely divisible satifying the conditions of Theorem \ref{t:classical_kolmogorov}, then 
\begin{eqnarray*}
d_{Kol}(\mu_{n},\Phi)
  &\leq&  C\sqrt{m_{4}(\mu)-3},
\end{eqnarray*}
where $m_{4}(\mu)$ denotes the fourth moment of $\mu$.

\end{cor}

Finally, we obtain a characterization for the normal distribution among infinitely divisible measures. This characterization was already observed in \cite{APA}.

\begin{cor}

Let $\mu$ be infinitely divisible with finite fourth moment $m_{4}(\mu)$. If $m_{1}(\mu)=0, m_{2}(\mu)=1$ and $m_{4}(\mu)=3$  then $\mu=\mathcal{N}(0,1)$.

\end{cor}

We now prove the main result of the section.

\begin{proof}[Proof of Theorem \ref{t:classical_kolmogorov}]
In what follows, for every random variable $Z$, $c_{j}(Z)$ will denote the $j$-th cumulant of $Z$ and $m_{j}(Z)$ its $j$-th moment. Let $S$ be a random variable with law $\mu$. There exist i.i.d. random variables $X_{1},...,X_{n}$, such that $S\stackrel{law}{=}X_{1}+\cdots  +X_{N}$. Consequently, 
$$S=\frac{1}{\sqrt{N}}\sum_{k=1}^{N}\sqrt{N}X_{k}.$$
Thus, by the Berry Esseen theorem we obtain the estimate 
\begin{eqnarray}\label{cka1}
d_{Kol}(\mu,\Phi)
  &\leq&  \frac{C}{\sqrt{N}}\E[|\sqrt{N}X_{1}|^3].
\end{eqnarray}

Using the additivity of cumulants (relation \eqref{aux1}) we obtain 

\begin{eqnarray}\label{eq:aux3}
1=c_{2}(S)=  \sum_{k=1}^{N}c_{2}(X_{k})=Nc_{2}(X_{1})=c_{2}(\sqrt{N}X_{1}),
\end{eqnarray}
and 
\begin{eqnarray}\label{eq:aux4}
Nc_{4}(S)=  N\sum_{k=1}^{N}c_{4}(X_{k})=N^2c_{4}(X_{1})=c_{4}(\sqrt{N}X_{1}).
\end{eqnarray}
Thus, applying the H\"older inequality, as well as relations $\eqref{cumulants}$,
\begin{eqnarray}\label{cka2}
\E[|\sqrt{N}X_{1}|^3]^2
  &\leq&  \E\left[(\sqrt{N}X_{1})^2\right]\E\left[(\sqrt{N}X_{1})^4\right]\nonumber\\
  &=&     c_{2}(\sqrt{N}X_{1})\left(c_{4}(\sqrt{N}X_{1})+3c_{2}(\sqrt{N}X_{1})^2\right)\nonumber\\
	&=&     Nc_{4}(X_{1})+3.
\end{eqnarray}
Finally, substituting \eqref{eq:aux3},\eqref{eq:aux4},\eqref{cka2} in \eqref{cka1}, and using relation \eqref{cumulants} we get
\begin{eqnarray*}
d_{Kol}(\mu,\Phi)
  &\leq&  \frac{C}{\sqrt{N}}\sqrt{Nc_{4}(X_{1})+3}\\
	&=&     C\sqrt{m_{4}(X_{1})-3+\frac{3}{N}}
\end{eqnarray*}
\end{proof}

The same result can be proven in a similar way for free random variables.

\subsection{Free convolution}

During this section, divisibility will be understood in the free sense, and the sum of random variables will be understood as the sum of free random varables in a suitable non-commutative probability space.

\begin{defi} Let  $\mu$ be a probability measure on $\mathbb{R}$.
\begin{enumerate} \item The measure $\mu$ is said to be $n$-divisible with respect to free convolution $\boxplus$  if it  is the $n$-fold free convolution of another probability measure. That is $\mu=\mu_n^{\boxplus n}$, for some $\mu_n$.  

\item The measure $\mu$ is called \emph{infinitely divisible} with respect to the free convolution $\boxplus$  if it is $n$-divisible for all $n\in\mathbb{N}$.
\end{enumerate}
\end{defi}

\begin{thm}\label{t:free_kolmogorov}
Let $\mu$ be a free $N$-divisible probability measure with variance 1 and mean 0. Assume that $m_{4}(\mu):=\int x^{4}\mu(dx)<\infty$. Then we have the following estimate for the error of the semi-circular approximation to $\mu$
\begin{eqnarray*}
d_{Kol}(\mu, F_w)
  &\leq&  2K\sqrt{m_{4}-2+\frac{2}{N}},
\end{eqnarray*}
where $F_w(x):=\int_{-\infty}^{x}\frac{1}{{2\pi}}\mathbf{1}_{(-2,2)}\sqrt{4-s^2} \text{d}s$, and $K$ is the universal constant from the free Berry-Esseen Theorem.
\end{thm}

\begin{rem}In Theorem  \ref{t:free_kolmogorov} we implicitly assume that $m_4-2+2/N$ is positive when $\mu$ is $N$-divisible. In Section \ref{kurt} we will show that this is the case and furthermore characterize when the mininum $m_4-3$ is achieved. 

\end{rem}

From Theorem \ref{t:free_kolmogorov} we can deduce the following result, analogous to Corollary \ref{maincorclass}  

\begin{cor}
Let $\{\mu_{n};n\geq1\}$ be a sequence of $N_{n}$-divisible probability measures with variance 1 and mean 0. If $N_{n}\rightarrow\infty$ and $\int s^{4}\mu_{n}(ds)\rightarrow 2$ as $n\rightarrow\infty$, then $\mu_{n}$ converges in distribution to a standard semicircle distribution. 

\end{cor}

Again, letting taking the limit as $N\to\infty$ we get Theorem \ref{T5}.

\begin{cor}\label{cor4}
Let $\mu$ be infinitely divisible satifying the conditions of Theorem \ref{t:free_kolmogorov}, then 
\begin{eqnarray*}
d_{Kol}(\mu, F_w)
  &\leq&  K\sqrt{m_{4}(\mu)-2},
\end{eqnarray*}
where $m_{4}(\mu)$ denotes the fourth moment of $\mu$.

\end{cor}

As for the classical case, we obtain a characterization for the semicircle distribution among infinitely divisible measures. This characterization was already observed in \cite{APA}.

\begin{cor}

Let $\mu$ be infinitely divisible with finite fourth moment $m_{4}(\mu)$. If $m_{1}(\mu)=0, m_{2}(\mu)=1$ and $m_{4}(\mu)=3$  then $\mu=\mathcal{S}(0,1)$.

\end{cor}

\begin{proof}[Proof of Theorem \ref{t:free_kolmogorov}]
Let $S$ be a random variable with law $\mu$. There exist i.i.d. random variables $X_{1},...,X_{n}$, such that $S\stackrel{law}{=}X_{1}+\cdots  +X_{N}$. Consequently, 
$$S=\frac{1}{\sqrt{N}}\sum_{k=1}^{N}\sqrt{N}X_{k}.$$
By the free Berry Esseen theorem we obtain
\begin{eqnarray}\label{ckaf1}
d_{Kol}(\mu,F_w)
  &\leq&  \frac{K}{\sqrt{N}}(\E[|\sqrt{N}X_{1}|^3]+\sqrt{\E[|\sqrt{N}X_{1}|^4]})
\end{eqnarray}
By H\"older inequality, $\E[|\sqrt{N}X_{1}|^3]^2\leq \E[(\sqrt{N}X_{1})^4]$, and thus
\begin{eqnarray}\label{ckaf1}
d_{Kol}(\mu,F_w)
  &\leq&  \frac{2K}{\sqrt{N}}\sqrt{\E[|\sqrt{N}X_{1}|^4]}).
\end{eqnarray}
Now, as in Theorem \ref{t:classical_kolmogorov}, we have $k_{2}(\sqrt{N}X_{1})=1$ and $Nk_{4}(S)=k_{4}(\sqrt{N}X_{1})$. Hence,  using \eqref{ckaf1}, as well as relations \eqref{freemomentcumulants}, we get
\begin{eqnarray*}
d_{Kol}(\mu,F_w)
  &\leq&  \frac{2K}{\sqrt{N}}  \sqrt{k_{4}(\sqrt{N}X_{1}) + 2k_{2}(\sqrt{N}X_{1})^2})\\
	&=&     2K\sqrt{m_{4}(X_{1})-2+\frac{2}{N}}.
\end{eqnarray*}
\end{proof}

\section{examples}

\begin{exa}[Poisson Distribution]
 Let $X_n$ be a random variable with distribution $Poiss(n)$, the random variable $Y_n=\frac{X_n-n}{\sqrt{n}}$ converges weakly to $N(0,1)$. $Y_n$ is infinitely divisible and $E(Y_n)=0, E(Y_n^2)=1, E(Y_n^4)=3+1/n$. Thus, we can apply Theorem \ref{T6} to quantify this approximation. 
$$d_{Kol}(Poiss(n),\Phi)\leq C\sqrt{1/n}.$$
\end{exa}

\begin{exa}[Compound Poisson Distribution]
More generally,  let $\mu$ be a  Compund Poisson distribution $Poiss(\lambda,\nu)$.  That is, the $n-$th cumulant of  $\mu$ is given by $c_n(\mu)=\lambda m_n(\nu)$.  If $\nu$ is centered with variance $1/\lambda$ then $m_1(\mu)=0, m_2(\mu)=1$ and $m_4(\mu)=\lambda m_4(\nu)+3$. Thus, Theorem \ref{T5} gives us,
$$d_{Kol}(Poiss(\lambda,\nu),\Phi)\leq C\lambda m_4(\nu).$$
In particular, if  $\lambda m_4(\nu)\to0$ then $Poiss(\lambda,\nu)\to \mathcal{N}(0,1).$
\end{exa}

\begin{exa}[Double Integrals]
 Let $\{F_n: n\geq1\}$ be a sequence living in a second chaos (see \cite{NuPe} for definitions) and suppose that $E[F_n^2]=1$. It is known that the random variable $F_n$ are infinitely divisible, see \cite{NoPo}. Thus, by Theorem \ref{T3} if $E[F_n^4]\to3$ then $F_n\to\mathcal{N}(0,1)$. Moreover, the Kolmogorov distance is bounded by
$$d_{Kol}(F_n,\Phi)\leq C\sqrt{E[F_n^4]-3}\approx 0.4748 \sqrt{E[F_n^4]-3}.$$
This shall be compared with the estimate given in Theorem 5.2.6 of \cite{NoPe4},
$$d_{Kol}(F_n,\Phi)\leq \frac{1}{\sqrt{6}}\sqrt{E[F_n^4]-3}\approx  0.4082\sqrt{E[F_n^4]-3}.$$
 \end{exa}

\begin{exa}[Log-normal]
The log-normal distribution $l(m,\sigma^2)$ with parameters $m$ and $\sigma^2>0$ is the distribution of the random variable $e^{m+\sigma Z}$. It is a well-known example of an infinitely divisible  distribution which is not determined by moments (see \cite{Th}). The moments of $e^{m+\sigma Z}$ are given by $E[e^{n(m+\sigma Z)}]=e^{nm+\frac{n^2\sigma^2}{2} }$. Thus, for $m(\sigma)=-1/2(log(e^{2\sigma^2}-e^{\sigma^2})$ the random variable  $Y =e^{m+\sigma Z}-e^{m+\frac{\sigma^2}{2} }$  is centered with variance one. Now,  $E(Y^4)-3=e^{4\sigma^2}+2e^{3\sigma^2}+3e^{2\sigma^2}-6\leq6e^{4\sigma^2}-6=6(e^{4\sigma^2}-1)$. Thus, for small $\sigma$ (since $C<1/2$), we have, 
$$d_{Kol}(l(m(\sigma),\sigma^2) ,\Phi)\leq 1/2\sqrt{6(e^{4\sigma^2}-1)}\approx\sqrt{6}\sigma.$$
This shows that Theorems \ref{T3} and \ref{T5} can be applied even if  $\mu$ is not determined by moments.
 \end{exa}

\begin{exa} [q-Gaussian]
The family of $q$-Gaussian distributions $G_q$ introduced by Bo\.{z}ejko and Speicher in \cite{BS} (see also the paper \cite{BKS} of Bo\.{z}ejko, K\"{u}mmerer and Speicher) interpolate between the normal ($q = 1$) and the semicircle ($q = 0$) laws. They are determined in terms of their moments by
$$ m_{2n+1}(G_q)=0,  ~~~~~~~ \quad m_{2n}(G_q)=\sum_{\pi\in P_2(2n)} q^{cross(\pi)}$$
where $P_2(2n)$ denotes de pair partitions of $\{1,2..,2n\}$ and $cross(\pi)$ is the number of crossings of $\pi$. 
In particular, $m_1(G_q)=0$, $m_2(G_q)=1$ and $m_4=2+q.$ It was proved in \cite{ABBL} that the $q$-Gaussian distributions are freely infinitely divisible for all $q\in[0,1]$.  Thus from Theorem \ref{T6} we get the estimate  $$d_{Kol}(G_q,F_w)\leq K\sqrt{q}.$$
 \end{exa}

\begin{exa}[Kesten-Mckay distribution]
Let $t>1/2$  and denote by $\mu(t)$ the so-called Kesten-Mckay distributions (see \cite{Kes,McK}) with density $$\frac{1}{2\pi}.\frac{\sqrt{4t-x^2}}{1-(1-t)x^2}, \quad |x|<2\sqrt{t}.$$
The first moments of  $\mu(t)$  are given by $m_1=0,m_2=1, m_3=0$ and $m_4=1+t$. Moreover, as proved in \cite{BN}, for $t<1$ the measure $\mu(t)$ is $1/(1-t)$-divisible. Thus if $1/(1-t)=N \in \mathbb{N}$ we can apply Theorem \ref{t:free_kolmogorov} to get the  inequality
\begin{equation}d_{Kol}(\mu(t), F_w)
  \leq  K\sqrt{m_4-2+\frac2N}=K\sqrt{1-t}.
\end{equation}
On the other hand for $t\geq1$  the measure $\mu(t)$ is infinitely divisible, in this case we apply Theorem \ref{T5} to get 
$$d_{Kol}(\mu(t), F_w)
 \leq K\sqrt{m_4-2}=K\sqrt{t-1}.$$
\end{exa}

\appendix

\section{Kurtosis and $n$-divisibility} \label{kurt}

The kurtosis of a probability distribution is a widely used quantity in
statistics and gives information about the shape of a given distribution.
Here we derive a simple necessary conditions for $n$-divisibility with
respect to the classical and free convolutions.  We use the first fourth cumulants with respect to these
convolutions.

The \textit{classical kurtosis} of a probability measure $\mu $ with finite
fourth moment is defined as%
\begin{equation*}
Kurt(\mu )=\frac{c_{4}(\mu )}{(c_{2}(\mu ))^{2}}=\frac{\widetilde{m}_{4}(\mu
)}{(\widetilde{m}_{2}(\mu ))^{2}}-3,
\end{equation*}%
where $c_{2}(\mu )$ and $c_{4}(\mu )$ are the second and fourth classical
cumulants, and $\widetilde{m}_{2}(\mu )$ and $\widetilde{m}_{4}(\mu )$\ the
second and fourth moments around the mean. It is always true that $Kurt(\mu
)\geq -2$.

\begin{prop}
\label{kurt}Let $\mu $ be a probability measure on $\mathbb{R}$ with finite
fourth moment. \ If $\mu $ is $n$-divisible in the classical sense then $Kurt(\mu )\geq -\frac{2}{n}$. Additionally, equality is achieved if and only if $$\mu=(\frac{1}{2} \delta_{1}+\frac{1}{2}\delta_{-1})^{*n}.$$
\end{prop}

\begin{proof}
Suppose $\mu $ is $n$-divisible. Let $\mu _{n}$ be such that $\underset{n
\text{ times}}{\underbrace{\mu _{n}\ast \cdots \ast \mu _{n}}}=\mu ,$ by
linearity of the cumulants we can see that 
\begin{equation*}
Kurt(\mu _{n})=\frac{\frac{1}{n}c_{4}(\mu )}{(\frac{1}{n}c_{2}(\mu ))^{2}}=n
\frac{c_{4}(\mu )}{(c_{2}(\mu ))^{2}}=nKurt(\mu )
\end{equation*}%
So $Kurt(\mu )=\frac{1}{n}Kurt(\mu _{n})\geq-\frac{2}{n}$, where we used the
fact that $Kurt\geq -2.$  Equality $Kurt=-2$ holds only when $\mu=1/2\delta_{1}+1/2\delta_{-1}$ proving the second part of the statement.
\end{proof}

The free kurtosis is defined similarly using the free cumulants instead of
the classical cumulants. That is, the \textit{free kurtosis} of a
probability measure $\mu $ is defined as 
\begin{equation*}
Kurt^{\boxplus }(\mu )=\frac{\kappa _{4}(\mu )}{(\kappa _{2}(\mu ))^{2}}=
\frac{\widetilde{m}_{4}(\mu )}{(\widetilde{m}_{2}(\mu ))^{2}}-2=Kurt(\mu )+1
\end{equation*}%
where $\kappa _{2}(\mu )$ and $\kappa _{4}(\mu )$ are the second and fourth
free cumulants. Notice that $Kurt^{\boxplus }(\mu )\geq -1.$

Using similar arguments as in Proposition \ref{kurt}, we obtain a sufficient
condition for free $n$-divisibility.

\begin{prop}
\label{crit kur2}Let $\mu $ be a probability measure on $\mathbb{R}$ with
finite fourth moment. \ If $\mu $ is $n$-divisible in the free sense
then $Kurt^{\boxplus }(\mu )\geq -1/n.$ Additionally, equality is achieved if and only if $$\mu=(\frac{1}{2} \delta_{1}+\frac{1}{2}\delta_{-1})^{\boxplus n}.$$
\end{prop}

\begin{proof}
Let $\mu $ be $n$-divisible in the free sense and $\mu _{n}$ be such that 
\begin{equation*}
\underset{n\text{ times}}{\underbrace{\mu _{n}\boxplus \cdots \boxplus \mu
_{n}}}=\mu .
\end{equation*}
Since $Kurt^{\boxplus }(\mu _{n})=nKurt^{\boxplus }(\mu )$ and $
Kurt^{\boxplus }(\mu _{n})\geq -1$, we get the result.  Again, since $Kurt=-2$ holds only when $\mu=1/2\delta_{1}+1/2\delta_{-1}$ we obtain the second part of the statement.
\end{proof}

\end{document}